\documentclass[12pt]{article}
\usepackage{e-jc}
\usepackage[square,sort,comma,numbers]{natbib}
\usepackage[]{algorithm2e}
\usepackage[english]{babel}
\usepackage{amsthm}
\usepackage[all]{nowidow}
\usepackage{url}
\usepackage{subcaption}			
\usepackage{caption}
\usepackage{pgfplots}
\usepackage{tikz}
\usetikzlibrary{positioning}

\usepackage{amsmath,amssymb}
\usepackage{graphicx}
\usepackage{hyperref}

\dateline{\today}{\today}{TBD}


\Copyright{  The authors. Released under the CC BY-ND license (International 4.0).}

\title{Counting Candy Crush Configurations}

\author{Adam Hamilton\\
\small ARC Centre of Excellence for Mathematical and Statistical Frontiers,\\[-0.8ex] \small School of Mathematical Sciences\\[-0.8ex]
\small University of Adelaide\\[-0.8ex] 
\small South Australia, Australia.\\
\small\tt adam.h.hamilton@adelaide.edu.au\\
\and
Giang T. Nguyen\\
\small ARC Centre of Excellence for Mathematical and Statistical Frontiers,\\[-0.8ex] \small School of Mathematical Sciences\\[-0.8ex]
\small University of Adelaide\\[-0.8ex] 
\small South Australia, Australia.\\
\small\tt giang.nguyen@adelaide.edu.au\\
\and
Matthew Roughan\\
\small ARC Centre of Excellence for Mathematical and Statistical Frontiers,\\[-0.8ex] \small School of Mathematical Sciences\\[-0.8ex]
\small University of Adelaide\\[-0.8ex] 
\small South Australia, Australia.\\
\small\tt matthew.roughan@adelaide.edu.au}

\pgfplotsset{compat=1.14}
\begin{document}

\maketitle


\begin{abstract}
A $k$-stable $c$-coloured Candy Crush grid is a weak proper $c$-colouring of a particular type of $k$-uniform hypergraph. In this paper we introduce a fully polynomial randomised approximation scheme (FPRAS) which counts the number of $k$-stable $c$-coloured Candy Crush grids of a given size $(m,n)$ for certain values of $c$ and $k$. We implemented this algorithm on Matlab, and found that in a Candy Crush grid with 7 available colours there are approximately $4.3\times 10^{61}$ 3-stable colourings. (Note that, typical Candy Crush games are played with 6 colours and our FPRAS is not guaranteed to work in expected polynomial time with $k=3$ and $c=6$.) We also discuss the applicability of this FPRAS to the problem of counting the number of weak $c$-colourings of other, more general hypergraphs.
\end{abstract}

\section{Introduction}
This paper solves a particular problem in enumerative combinatorics inspired by the popular mobile-phone game Candy Crush, which is a match-3 puzzle game \cite{CandyCrush}. The objective of the game is to swap coloured pieces of ``candy" on a game board to match three or more of the same colour. Once these candies have been matched, they are eliminated from the board and replaced with new ones, which potentially creates further matches. This chain reaction will continue until there are no more matchings. The longer this chain reaction lasts the more points are scored.\\
\newline\noindent
In this paper, we seek to approximate the size of the \textit{state space} of this particular game, that  is, the number of ways that we can assign colours to the individual candies such that no matchings occur. We formulate this problem in terms of \textit{hypergraphs}, where the individual candies are the vertices and the potential matchings are the hyperedges. We find that approximating the size of the state space of Candy Crush is a special case of the more general problem of estimating the weak chromatic polynomial of a hypergraph \cite{ZhangHypergraph}.\\
\newline\noindent
We describe a randomised algorithm and the conditions (Inequality (\ref{eqnconditions})) for it to be a \textit{fully polynomial randomised approximation scheme} (FPRAS), an algorithm that provides an estimate with arbitrarily small error of the chromatic polynomial in randomised polynomial time. For Candy Crush, the conditions can be summarised as: our algorithm is an FPRAS if the candies can take at least 7 colours, which is slightly larger than the default case in the game (6 Colours).\\
\newline\noindent
This FPRAS is based on the Multilevel Splitting algorithm (MSA) used in rare event analysis \cite{Lagnoux}. The idea behind the Multilevel Splitting algorithm is to simulate a rare event in a large state space by dividing the state space into a sequence of nested sub-spaces each of which are not rare events themselves. This approach has been quite successful in estimating the chromatic polynomials of simple graphs \cite{Slava,ApproximationAlgorithms}. However, there has been no work of which we are aware using Multilevel Splitting algorithms to approximate the weak chromatic polynomials of hypergraphs, or for counting Candy Crush configurations.\\
\newline\noindent
In order for our algorithm to be an FPRAS, we must be able to approximately uniformly sample elements from the set of Candy Crush configurations in randomised polynomial time. Most Multilevel Splitting algorithms use Markov Chain Monte Carlo methods to perform these samplings. These techniques are generally used because it is simple to prove that they provide approximate uniform samplings. In our paper, we use the \textit{Moser-Tardos} algorithm \cite{MoserTardos} to perform uniform samplings. The reason we use this method is because, given the best possible conditions, the Markov Chain Monte Carlo algorithms run in time $O(n^2\log(n))$ \cite{Grid}, whereas the Moser-Tardos algorithm runs in expected linear time \cite{MoserTardos}. Proving that the Moser-Tardos algorithm gives uniformly distributed samples is more challenging than proving uniformity of Markov Chain Monte Carlo. Guo \textit{et al.}~\cite{UniformMT} studied the output distribution of the Moser-Tardos algorithm, and showed that if a constraint satisfaction problem is \textit{extremal} then the output of the Moser-Tardos algorithm applied to said problem will be uniformly distributed. Unfortunately, finding weak proper colourings of hypergraphs is not extremal and we must come up with a different proof that Moser-Tardos sampling from the state space of Candy Crush is uniform. To our knowledge this is the first example of a constraint satisfaction problem that has a proven uniform distribution from the Moser-Tardos algorithm.\\
\newline\noindent
We implemented our estimation algorithm using Matlab. We found, in a standard in-game level of Candy Crush in which there are 9 rows, 9 columns, and 6 available colours that, out of the $6^{81}\approx1.072\time 10^{63}$ possible colourings, approximately $0.040\%$ of these colourings form valid levels in Candy Crush. Although there was no theoretical guarantee that the algorithm should produce an estimate in polynomial time, it did not take Matlab long to produce an estimate. This result is not in and of itself difficult to obtain. For $9\times9$ grids, the probability that a randomly selected 6-colouring forms a weak proper $c$-colouring is high enough for us to use Monte Carlo simulation to obtain an accurate approximation. The significance of our algorithm is that it runs in randomised polynomial time, whereas Monte Carlo simulation runs in exponential time. This allows us to compute results for larger grids.


\section{Problem formulation}
\label{problem formulation}
The game Candy Crush is described in more detail in \cite{CandyCrush}. We define the Candy Crush grid, the lattice in which the game is played, as a hypergraph whose vertices are a subset of a square lattice and whose hyperedges are sets of $k$ vertices that appear consecutively in a single row or column. These hypergraphs are designed to represent individual in-game levels in Candy Crush, which is played on a rectangular array containing different colours of candies. In an actual game of Candy Crush, the number of columns and rows, $n$ and $m$, are between 1 and 9 and $k=3$.\\
\newline\noindent 
More precisely, a Candy Crush grid is a $k$-uniform hypergraph, $H=(V,E)$, parameterised by natural numbers $m, n$, and $k$, where the set of vertices, $V$, of the Candy Crush grid is a subset of $\{1,\cdots,n\}\times\{1,\cdots,m\}$, and the hyperedges, $E$, of the Candy Crush grid are the subsets of $V$, such that each $e\in E$ is of the form $\{(i,j),(i,j+1),\cdots,(i,j+k-1)\}$ or $\{(i,j),(i+1,j),\cdots,(i+k-1,j)\}$. See Figure \ref{fig:hypergraphfig} for an example of a Candy Crush grid and its hypergraph representation. \\
\begin{figure}[h]
    \centering

    \begin{subfigure}[t]{0.35\textwidth}
    \label{hypergraphfig,a}
        \includegraphics[width=0.8\textwidth]{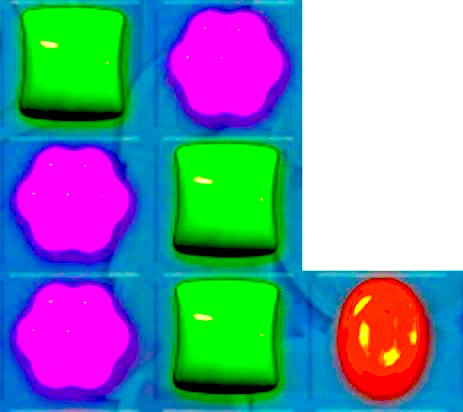}
        \caption{}
        \label{fig:MTiter}
    \end{subfigure}
        ~ 
    \begin{subfigure}[t]{0.5\textwidth}
    \label{hypergraphfig,b}
        \includegraphics[width=0.8\textwidth,trim={4cm 11cm 3cm 10cm},clip]{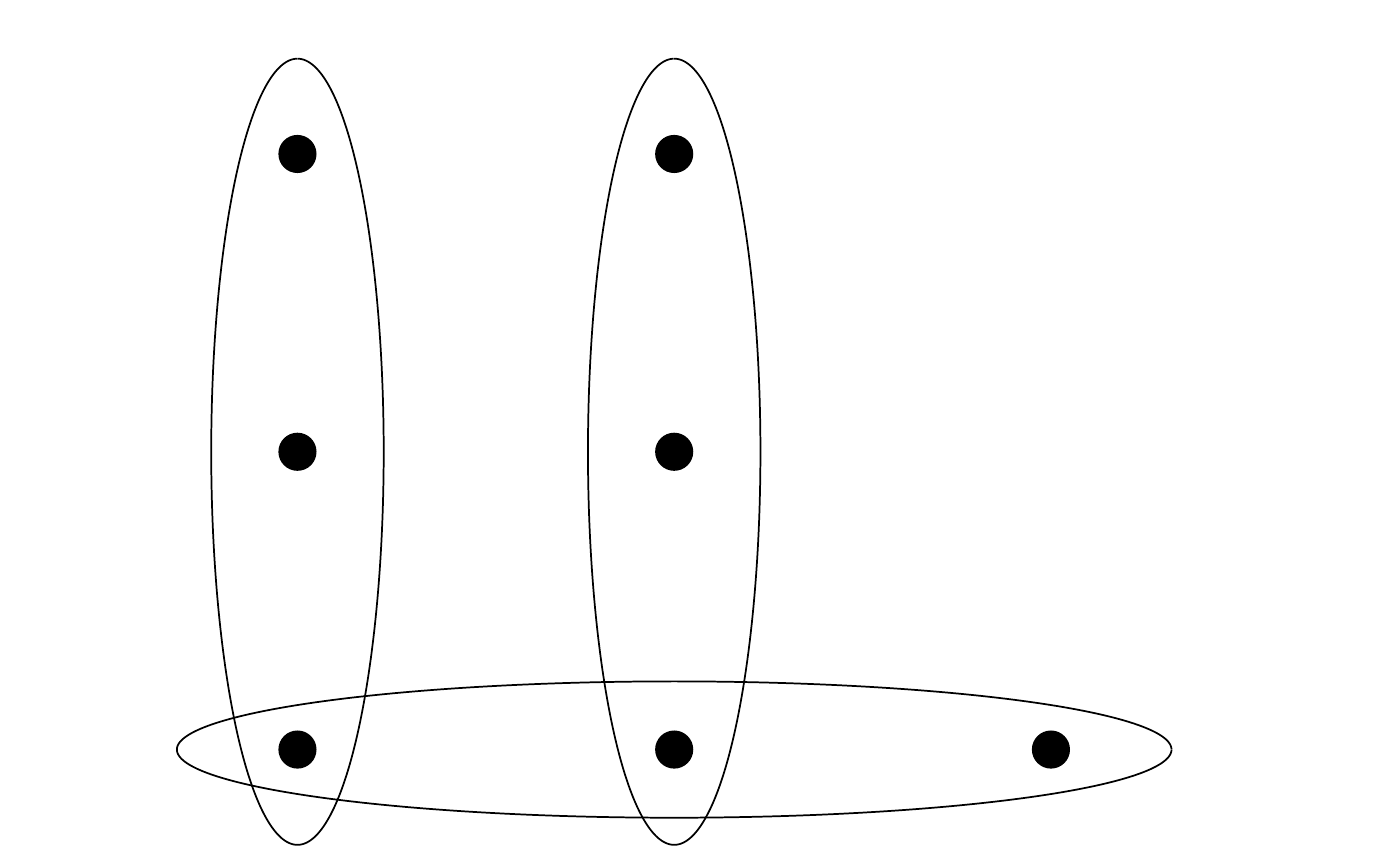}
        \caption{}
        \label{fig:orginalgrid}
    \end{subfigure}
    
    \caption{An example of a Candy Crush grid shown in how it would appear in the game (a), and the corresponding hypergraph (b). }\label{fig:MTvsCC}
\label{fig:hypergraphfig}
\end{figure}

\noindent There are two types of proper hypergraph colouring, they are called the \textit{strong} and \textit{weak} colourings. A \textit{weak proper $c$-colouring} of a hypergraph $H$ is a mapping $\phi:V\rightarrow\{1,..,c\}$ such that $|\{\phi(v):v\in e\}|>1, \forall e\in E$. That is, no hyperedge in $H$ is monochromatic. A \textit{strong proper $c$-colouring} of a hypergraph $H$ is a mapping $\phi:V\rightarrow\{1,..,c\}$ such that $|\{\phi(v):v\in e\}|=|e|, \forall e\in E$. That is, every vertex in each hyperedge of $H$ is assigned a unique colour within the hyperedge. The number of weak proper $c$-colourings of $H$ is given by the function $P(H,c)$. This function is called the \textit{weak chromatic polynomial} and is polynomial in $c$, with degree $|V|$, the number of vertices in the hypergraph \cite{ZhangHypergraph}.\\
\newline\noindent
Weak colourings relate to match-3 games because in a typical match-3 game, the board is filled with various tiles that one must shift, select or rotate to form chains of 3 or more identical elements. For instance, the vertices in the hypergraph correspond to the candies in Candy Crush and the hyperedges correspond to potential matchings. A weak proper $c$-colouring of a Candy Crush grid corresponds to a game of Candy Crush in which no sets of three candies have been matched. In the game if no elements have been matched then the game is called \textit{stable}, because no further changes will happen on the board until the next move from the player. From here on the terms \textit{stable} and \textit{weak proper $c$-colouring} will be used interchangeably.\\

\subsection{The probabilistic set-up}
In a hypergraph there are a total of $c^{|V|}$ possible $c$-colourings. For large enough values of $c$, a non-zero proportion, $\ell$, of these colourings are weak proper $c$-colourings of $H$. In other words, $\ell$ is the probability that a $c$-colouring, selected uniformly at random, is a weak proper $c$-colouring. Finding the value of $\ell$ is equivalent to finding the value of the weak chromatic polynomial $P(H,c)$, since $P(H,c)=\ell c^{|V|}$.\\
\newline\noindent
Exact calculation of $P(H,c)$ is presumed to be intractable. This is because an undirected graph $G$ is an example of a 2-uniform hypergraph and evaluating the chromatic polynomial of a simple graph is known to be $\#P$-hard \cite{ApproximationAlgorithms}. Thus the problem of evaluating $P(H,c)$ contains the problem of evaluating $P(G,c)$ as a sub-problem, and is therefore also $\#P$-hard. There are no known poly-time deterministic algorithms that can solve such problems. We are operating under the assumption that calculating $P(H,c)$ is still in $\#P$-hard even when $H$ is restricted to Candy Crush grids.\\
\newline\noindent
While exact computation of $P(H,c)$ may be intractable, one could still consider estimating $P(H,c)$. The simplest way to estimate $\ell$ (and thus $P(H,c)$) would be to use Monte Carlo simulation and is summarised in Algorithm \ref{alg:Monte-Carlo}:\\

\begin{algorithm}[H]
\label{alg:Monte-Carlo}
 \KwData{A hypergraph $H=(V,E)$, a natural number $c$, and a number of iterations $N$}

 \KwResult{An unbiased estimator of $\ell$}
 set $count=0$

 \For{$i=1,...,N$}{
 Uniformly sample a colouring of $H$ by uniformly assigning every vertex $v\in V$ a value from the set $\{1,..,c\}$\\
 \If{\text{There is no monochromatic hyperedge $e\in E$}}{$count=count+1$}
 
}
\textbf{return}: $\hat{\ell}_{mc}=count/N$, which is an unbiased estimator of $\ell$.
\vspace{0.5cm}
 \caption{The Multilevel Splitting algorithm used to estimate $\ell$. The number of samples in each level, $T_{\text{sample}}$, is determined in Section \ref{Computational Complexity} as a function of the inputs $|V|$, $\varepsilon$, and $\delta$.}
\end{algorithm}

\subsection{Rare-event Monte Carlo}
\label{rareevent}
 The while loop in Algorithm \ref{alg:Monte-Carlo} has an expected exponential number of iterations $N$ in order to produce a single sample. This means that producing estimates using Monte Carlo simulation would have a similar complexity. Using Monte Carlo simulation will give us an unbiased estimator of $\ell$, called $\hat{\ell}_{MC}$, which is a binomially distributed random variable with mean $\ell$ and variance $\ell(1-\ell)/N$. Consider the relative error of $\hat{\ell}_{MC}$, defined by 

\begin{equation}
\mbox{RE}(\hat{\ell}_{MC})=\frac{\sqrt{var(\hat{\ell}_{MC})}}{\ell}=\frac{\sqrt{(1-\ell)}}{\sqrt{\ell N}},
\label{RelativeError}
\end{equation}

\noindent for fixed $c$. Since $P(H,c)$ is of degree $|V|$, the value of $P(H,c)$ can be asymptotically bounded above and below by two exponential functions. Thus, in order to ensure that $\mbox{RE}(\hat{\ell}_{MC})$ is bounded above, the number of samples $N$ must grow exponentially with the number of vertices $V$ in $H$. Consequently, even for moderate-sized hypergraphs Algorithm \ref{alg:Monte-Carlo} will be unable to yield a useful estimate. In order to properly estimate $P(H,c)$ we need another approach. 

\section{The Algorithm}
\label{algorithms}
\subsection{FPRAS}
\label{def:FPRAS}
In general, counting problems involve computing some function $f:\Sigma^*\rightarrow\mathbb{N}$, where $\Sigma$ is some finite alphabet used to encode a problem and $\Sigma^*$ is the set of finite sequences whose elements are in $\Sigma$. In our case, the problem of counting the number of $k$-stable $c$-colourings of an $m \times n$ Candy Crush grid can be encoded by the triple $(p_{m,n},k,c)$, where $m,n,c$ and $k$ are natural numbers, $p_{m,n}\in V$, the set of vertices in the Candy Crush grid, and $f$ is the weak chromatic polynomial $P(H,c)$.\\
\newline\noindent
A \textit{fully polynomial randomised approximation scheme} (FPRAS) is a randomised algorithm such that, given an input $x\in\Sigma^*$, an error parameter $\varepsilon>0$, and a confidence parameter $0<\delta<1$, we can, in bounded time polynomial in $|x|$, $\varepsilon^{-1}$ and $\log(1/\delta)$, compute an estimate $\hat{f}(x)$ where
$$
 \text{Pr}\bigg\{(1-\varepsilon)f(x)\leq \hat{f}(x) \leq(1+\varepsilon)f(x)\bigg\}\geq 1-\delta .
$$
The algorithm we use to estimate $P(H,c)$ is not strictly an FPRAS, because it runs in probabilistic polynomial time instead of deterministic polynomial time. This means there is a probability that the algorithm will run in super-polynomial time, but this probability is presumed to be negligible based on empirical results from Catarata \textit{et al.}~\cite{Catarata}.

\subsection{Multilevel Splitting algorithm}
\label{level selection}
Here we estimate $P(H,c)$ in randomised polynomial time using a \textit{Multilevel Splitting algorithm}. First we impose a lexicographic ordering on the vertices in the Candy Crush grid, to make the algorithm easier to describe. We say vertex $x=(i,j)$ is lexicographically higher than $x'=(i',j')$ iff $i>i'$, or both $i=i'$ and $j>j'$.\\
\newline\noindent
The Multilevel Splitting algorithm is used to estimate the probability of a rare event by splitting the event up into a sequence of nested sub-events, none of which are rare themselves (see Algorithm \ref{alg:MSA}).\\

\begin{algorithm}[H]
\label{alg:MSA}
 \KwData{A set $Y$, a sequence of subsets $Y=Y_0\supseteq Y_1 \supseteq .... \supseteq Y_T=Y^* $, and a fitness function $S:Y\rightarrow \mathbb{R}$ such that $S(y_i)\geq S(y_{i+1})$ for all $y_i \in Y_i$ and $y_{i+1}\in Y_{i+1}$}

 \KwResult{An unbiased estimator $p$ of $|Y^*|/|Y|$ }
 \For{$t=1,\cdots,T$}{
  calculate an estimator $\hat{c}_t$ for the conditional probabilities $c_t=\text{Pr}(S(Y)\leq S(Y_t)|S(Y)\leq S(Y_{t-1}))$, the probability that a randomly selected member of $Y_t$ is a member of $Y_{t+1}$.\

 }
\textbf{return}: 
$$
\hat{\ell}= \prod_{t=1}^{T}\hat{c}_t.
$$
 \caption{The general Multilevel Splitting algorithm.}
\end{algorithm} 
\vspace{0.5cm}
\noindent Recall that $\ell$ is the probability that a random colouring of a Candy Crush grid $H=(V,E)$ is a weak proper $c$-colouring. We estimate $\ell$ by letting $Y$ be $\{1,...,c\}^{|V|}$, the set of all possible colourings, $Y^*$ be the set of weak proper $c$-colourings of $H$, and the intermediate subsets $Y_t$ be the set of colourings in which the first $t$ vertices do not belong to any monochromatic hyperedges. Thus, the number of levels in the algorithm, $T$, is equal to the number of vertices in the Candy Crush grid, $|V|$, and is bounded above by $mn$.

\subsection{Moser-Tardos algorithm }
\label{Moser-Tardos algorithm}
For the Multilevel Splitting algorithm to perform as an FPRAS, we must be able to calculate the estimators $\hat{c}_t$ in polynomial time. In this paper, we use Monte Carlo estimation to perform the sampling used to calculate $\hat{c}_t$. This means that we need to sample from the set $Y_t$ in polynomial time.\\
\newline\noindent
The most straightforward means to sample from $Y_t$ would be to use Markov Chain Monte Carlo methods, such as \textit{heat-bath dynamics} \cite{Grid} or \textit{Glauber dynamics} \cite{Grid}. These methods have been proven to work in the case that $c\geq6$, producing approximately uniform samples in time $O(|V|^2\log(|V|))$ \cite{Grid}. In this paper, we use a Las Vegas algorithm called the Moser-Tardos algorithm \cite{MoserTardos}. This algorithm samples from $Y_t$ in randomised time, $T_{\text{random}}$, where $\mathbb{E}(T_{\text{random}})\leq O(|V|)$ under certain conditions. It is worth noting that the expected run-time of the Moser-Tardos algorithm is linear in $|E|$, the number of hyperedges, but since in a Candy Crush grid $|E|=O(|V|)$, we can claim that $\mathbb{E}(T_{\text{random}})=O(|V|)$.\\
\newline\noindent
The crucial difference between Monte Carlo and Las Vegas algorithms is as follows. In this context, a Monte Carlo algorithm would take an element of the set of colourings in which the first $t$ vertices do not belong to any monochromatic hyperedges $Y_t$, perform a random walk over this set and, after a fixed number of iterations, stop at an element of $Y_t$ which is distributed arbitrarily close to the uniform distribution. A Las Vegas algorithm, on the other hand, would perform a random walk over the elements of $Y$ and, after a randomly distributed number of iterations, stop at an element of $Y_t$ which is exactly uniformly distributed. How this random number of iterations is determined depends on the choice of algorithm and choice of input.\\
\newline\noindent
In general, the Moser-Tardos algorithm is specified over a set of mutually independent random variables and a set of events which are Boolean functions on the values of these random variables. When using the Moser-Tardos algorithm to solve constraint satisfaction problems, there is some degree of freedom as to how these sets are defined. In Section \ref{uniformity proof}, we show that the choice of variables has a non-trivial effect on the efficiency and output-distribution of the algorithm. Algorithm \ref{alg:Moser-Tardos} describes the Moser-Tardos algorithm used to find a weak proper $c$-colouring $y\in Y$ in the hypergraph $H=(V,E)$.\\

\begin{algorithm}[H]
\label{alg:Moser-Tardos}
 \KwData{A hypergraph $H=(V,E)$ and a natural number $c$}

 \KwResult{A weak proper $c$-colouring of $H$ }
 Randomly uniformly assign every vertex $v\in V$ a value from $\{1,\cdots,c\}$\\
 \While{There is some monochromatic $e\in E$}{
 Randomly uniformly assign every vertex $v\in e$ a value from $\{1,\cdots,c\}$\\
 }
\textbf{return }: The assignment of elements of $\{1,\cdots,c\}$ to elements of $V$

 \caption{The Moser-Tardos algorithm used to find a weak proper $c$-colouring of a hypergraph $H$.}
\end{algorithm} 
\vspace{0.5cm}
\noindent Algorithm \ref{alg:Moser-Tardos} samples weak proper $c$-colourings of a Candy Crush grid in linear expected time only if $c$ and $k$ obey certain constraints. Otherwise, the algorithm, on average, creates more monochromatic hyperedges than it removes in any iteration. Let $W$ be the set of points $(c,k)\in\mathbb{N}^2$ such that Algorithm \ref{alg:Moser-Tardos} samples a $k$-stable $c$-colouring of a Candy Crush of arbitrary size in expected superlinear time. Let  $\partial W$ be the boundary of this set, this boundary is called the \textit{metastable equilibrium} \cite{Catarata}. Currently, finding an analytical expression for the metastable equilibrium is an open problem. In Section \ref{complexity} we define $W'$, a subset of $\mathbb{N}^2$ such that $W\subseteq{W'}$. This is important because it provides us with a sufficient but non-necessary condition for Algorithm 3 to work as an FPRAS. \\

\begin{algorithm}[H]
\label{alg:MSAforCC}
 \KwData{A Candy Crush grid $H=(V,E)$, with parameter $k$, natural number $c>2$, an arbitrary constant $\lambda$, an error parameter $\varepsilon>0$, a confidence parameter $0<\delta<1$, and a lexicographic ordering of $V$}

 \KwResult{An unbiased estimator $\hat{\ell}$ of $P(H,c)/c^{|V|}$ }
 \For{$t=1,...,|V|$}{
 set $count_t=0$\\
 \For{$i=1,...,T_{\text{sample}}$}{
 Uniformly sample an element $y_i$ of $Y_t$:
 \begin{itemize}
     \item  Sample a weak proper $c$-colouring of $H_t=(\{v_1,v_2,...,v_t\},E')$, the sub-hypergraph of $H$ restricted to the first $t$ vertices of $V$ and with all possible edges. This sampling can be done using Algorithm \ref{alg:Moser-Tardos}.
     \item Assign colours from $\{1,\cdots,c\}$ to the other $(|V|-t)$ vertices of $H$ uniformly at random.
 \end{itemize}
 \If{$y_i\in Y_{t+1}$}{$count_t=count_t+1$}
 
 }
  calculate an estimator $\hat{c}_t$ for the conditional probabilities $\hat{c}_t=count_t/T_{\text{sample}}$ \;

 }
\textbf{return}: 
$$
\hat{\ell}= \prod_{t=1}^{T}\hat{c}_t.
$$
 \caption{The Multilevel Splitting algorithm used to estimate $\ell$. The number of samples in each level, $T_{\text{sample}}$, is determined in Section \ref{Computational Complexity} as a function of the inputs $|V|$, $\varepsilon$, and $\delta$.}
\end{algorithm}

\section{Computational Complexity}
\label{Computational Complexity}
We put the pieces together in Algorithm \ref{alg:MSAforCC}, which is an FPRAS for estimating the weak chromatic polynomial of certain Candy Crush grids, and is the main result of this paper. We first demonstrate that, assuming sampling can be done in polynomial time, Algorithm \ref{alg:MSAforCC} satisfies the definition of FPRAS given in \cite{ApproximationAlgorithms}. Let $\varepsilon$ be the error parameter of the FPRAS and let $\delta$ be the confidence parameter. For $t=1,2,...,|V|$, where $|V|$ is the total number of levels in the MSA, our aim is to estimate $c_t$ within a factor of $(1\pm \varepsilon/3|V|)$ with probability $\geq1-\delta$, since $(1\pm\varepsilon/3|V|)^{|V|}\geq(1+\varepsilon)$.\\
\newline\noindent
In order to prove that Algorithm \ref{alg:MSAforCC} is an FPRAS, we need to establish the following lemmas. First we establish crude upper and lower bounds for $c_t$. These are necessary in order to show that $c_t$ can be estimated to a desired level of accuracy with a number of samples that is independent of the size of the Candy Crush grid.\\

\begin{lemma}\label{Bounding Levels}
There exist real numbers $a=1/2$ and $b=1-(c-2)/(c^{k-1})$ such that $a\leq c_t \leq b$ for all $t \in \{1,2,\cdots,|V|-k\}$ and for all $|V|\in \mathbb{N}$.
\end{lemma}

\begin{proof}
Recall that $Y_t$ is the set of Candy Crush grids in which the first $t$ candies do not belong to any monochromatic $k$-uniform hyperedges. By definition $c_t=|Y_{t+1}|/|Y_{t}|$. Form this it can be shown that $c_t=1-|Y_t\setminus Y_{t+1}|/|Y_{t}|$. It is helpful to consider the set $Y_t\setminus Y_{t+1}$, the set of all colourings in which the first monochromatic hyperedge begins with the vertex at position $t+1$.\\
\newline\noindent
In our paper the coordinates $(i,j)$ refer to points with integer coordinates in the upper right quadrant of the Cartesian plane. In order to find a lower bound for $c_t$ we construct a bijective function between $Y_t\setminus Y_{t+1}$ and a subset of $Y_{t+1}$. This implies that $|Y_{t+1}|\geq |Y_t\setminus Y_{t+1}|$ and therefore that $c_t\geq 1/2$. This function takes an element of $Y_{t+1}$ and creates a monochromatic hyperedge that begins with the lexicographically lowest vertex in a monochromatic $k$-mer in $Y_t\setminus Y_{t+1}$.\\
\newline\noindent
One bijective function $f:Y_t\setminus Y_{t+1} \rightarrow Y_{t+1}$ that works is the function defined by the following process:
\begin{enumerate}
    \item Take an element $G'$ of $Y_t\setminus Y_{t+1}$, then $G'$ must have either one or two monochromatic hyperedges. If $G'$ has only one monochromatic hyperedge call this $e$. If $G'$ has two monochromatic hyperedges, call these $e_1$ and $e_2$.
    \item Take the lexicographically highest vertex, $x$, in $e$ (if $G'$ has 2 monochromatic hyperedges then take the 2 lexicographically highest vertices of both hyperedges, $e_1$ and $e_2$, call then $x_1$ and $x_2$) this/these vertices will have some colour $\ell\in \{1,\cdots,c\}$.
    \item $f(G')$ returns $G'$ with one modification: it changes the colour of the vertex $x$ (or $x_1$ and $x_2$) from  $\ell$ to $\ell+1 (\text{mod } c)$. Note that this modified $G'$ is in $Y_{t+1}$. 
\end{enumerate}

\noindent It is clear to see that if we consider only the subset of $Y_{t+1}$ that is the image of $Y_t\setminus Y_{t+1}$ under $f$, then $f$ is bijective, since $f$ consists of a single addition modulo $c$.\\
\newline\noindent
In order to find an upper bound for $c_t$, we need to find an upper bound for $|Y_{t+1}|/|Y_{t}|$, or a lower bound for $|Y_t\setminus Y_{t+1}|$. If we consider an element of $Y_t\setminus Y_{t+1}$, the lexicographically first monochromatic hyperedge is either vertical or horizontal (w.l.o.g. we assume horizontal). Let $x$ be the lexicographically first vertex in a monochromatic hyperedge, with coordinates $(i,j)$, then a sufficient but non-necessary condition for $x$ to be the lexicographically lowest candy of the lexicographically lowest monochromatic hyperedge is it cannot be the same colour as the vertices with coordinates $(i-1,j)$ and $(i,j-1)$. It should be noted that it is possible for $x$ to be the lexicographically lowest vertex in the lexicographically lowest monochromatic hyperedge if $x$ is the same colour as $(i-1,j)$ and $(i,j-1)$ or both. But we are interested in finding a lower bound for $|Y_t\setminus Y_{t+1}|$. For this reason, we are deliberately under counting the elements of $|Y_t\setminus Y_{t+1}|$ and only considering the conditions in which $x$ is guaranteed to be the lexicographical lowest candy in the lexicographically lowest monochromatic hyperedge. Therefore, in this restricted setting, $x$ can take on $c-2$ possible colours. Since $x$ is an element of $Y_t\setminus Y_{t+1}$ and, in the interests of deriving a lower bound, we are currently considering horizontal monochromatic hyperedges, then the $k-1$ vertices to the right of $x$ must all be the same colour as $x$. This gives a lower bound for $|Y_t\setminus Y_{t+1}|$ of $|Y_t|(c-2)/c^{k-1}$, the probability that $x$ is the lexicographically lowest candy in a horizontal monochromatic hyperedge that can take on $c-2$ colours.\\
\newline\noindent
Setting $a=1/2$ and $b=1-(c-2)/(c^{k-1})$, this completes the proof of Lemma \ref{Bounding Levels}.
\end{proof}

\begin{theorem}\label{thm:FPRAS}
Assuming that we can sample uniformly from all the sets $Y_t$ in polynomial time, Algorithm \ref{alg:MSAforCC} is an FPRAS for approximating the number of $k$-stable Candy-Crush grids.
\end{theorem}

\begin{proof}

\noindent The bounds for $c_t$ from Lemma \ref{Bounding Levels}, imply that $c_t$ does not correspond to a rare event for any values of $t$, that is, $c_t$ does not decrease super polynomially towards zero as $t$ increases. This implies that the number of samples needed to estimate $c_t$ to a specified degree of accuracy remains constant regardless of the size of the Candy-Crush grid and the value of $t$. We use Chernoff's inequality to properly estimate the number of necessary samples, $T_{\text{sample}}$, from the sets $Y_t$.\\ 
\newline\noindent
Let $T_{\text{sample}}$ denote the number of samples used to calculate the estimator $\hat{c}_{t}$. From Section \ref{rareevent} we note that $\hat{c}_t$ is unbiased, \textit{i.e}, $\mathbb{E}[\hat{c}_t]=c_t$. Our aim is to estimate $c_t$ within a factor of $(1\pm \varepsilon/3|V|)$ with probability $\geq1-\delta/|V|$. Using Chernoff's inequality, we show that if $T_{\text{sample}}=54((|V|/\varepsilon)^2\log{(2|V|/\delta)}$ we can estimate $c_t$ with the desired levels of accuracy and confidence. For the sake of brevity $T$ and $T_{\text{sample}}$ will be used interchangeably. To find a workable value for $T$ we need to satisfy the inequality,

$$
\text{Pr}\big\{|T\hat{c}_t-Tc_t|>Tc_t\varepsilon/3|V|\big\}\leq 2e^{-(\varepsilon/3|V|)^2Tc_t/3}=\delta/|V|.
$$
So we take the equation 
$$
2e^{-(\varepsilon/3|V|)^2Tc_t/3}=\delta/|V|,
$$
and rearrange to obtain
$$
Tc_t=27(|V|/\varepsilon)^2\log(2|V|/\delta).
$$
From Lemma \ref{Bounding Levels}, we know that $1/2\leq c_t$, and so an upper bound for $T$ that satisfies  $\text{Pr}\big\{|T\hat{c}_t-Tc_t|>Tc_t\varepsilon/3|V|\big\}\leq \delta/|V|$ is $T=54(|V|/\varepsilon)^2\log(2|V|/\delta)$. Using this value of $T$ we have

$$
\text{Pr}\big\{|\hat{c}_t-c_t|>c_t\varepsilon/3|V|\big\}<\delta/|V|.
$$
From this, we can conclude with probability $\geq 1-\delta$, for all $t$, that we have
$$
c_t(1-\varepsilon/3|V|)\leq \hat{c}_t\leq c_t(1+\varepsilon/3|V|).
$$
Let $\hat{\ell}=\prod_{t=1}^{|V|}\hat{c}_t$. Since $(1\pm\varepsilon/3|V|)^{|V|}\geq(1\pm\varepsilon)$, we have
$$
\text{Pr}\big\{(1-\varepsilon)\ell\leq\hat{\ell}\leq(1+\varepsilon)\ell\big\}\geq 1-\delta.
$$

\noindent The last step is prove that this algorithm runs in time polynomial in $|V|$, $\varepsilon^{-1}$, and $\log(\delta^{-1})$. There are $|V|$ levels, $O((|V|/\varepsilon)^2\log{(2|V|/\delta)}$ samples at each level, and since it takes $O(f(m,n,c,k))$ time to perform each sample, where $f$, an upper bound to the number of basic operations required to sample from $Y_t$, is a function still to be determined. An upper bound for the time complexity of this FPRAS is $O(mn\times ((mn/\varepsilon)^2\log{(2mn/\delta))}\times f(m,n,c,k))$, which, assuming $f$ can be bounded by some polynomial, is a polynomial run-time.
\end{proof}

\noindent Lemma \ref{Bounding Levels} proves that the number of samples needed to estimate $c_t$ to a desired level of accuracy, $\varepsilon$, is independent of $m,n$, and $t$.\\
\newline\noindent
Note that, Chernoff's inequality gives a very large upper bound for $T_{\text{sample}}$. In our proof of Theorem \ref{thm:FPRAS}, we show that $O(T^3\log(T))$ total samples are required \textit{i.e}, $O(T^2\log(T))$ samples at each level. In \cite{jerrumbook} the author showed that only $O(T^2)$ total samples are needed to estimate $\hat{c}_t$ (there are $O(T)$ samples at each level). This means that Algorithm \ref{alg:MSAforCC} can be successfully implemented with much fewer iterations than the ones derived in this paper. Since our paper focuses on proving the existence of an FPRAS, we have used the bounds derived using Chernoff's inequality.\\

\subsection{Complexity of sampling}

By Theorem \ref{thm:FPRAS}, if sampling weak proper $c$-colourings of a Candy Crush grid can be done in polynomial time then, the Multilevel Splitting Algorithm \ref{alg:MSAforCC} is an FPRAS. In Section \ref{Moser-Tardos algorithm} we introduced $W$, the set of values of $(c,k)$ such that the Moser-Tardos algorithm will sample a $k$-stable $c$-colouring of an arbitrary sized Candy Crush grid in expected linear time. In this section, we derive a set $W'\subset \mathbb{N}^2$ such that $W'\subseteq W$. This lets us find some values of $c$ and $k$ such that the Moser-Tardos algorithm is guaranteed to run in randomised linear time.\\
\newline\noindent
In 2009, Moser and Tardos \cite{MoserTardos} proved that if the existence of a solution to a particular constraint satisfaction problem (CSP) is guaranteed to exist due to the \textit{Lov\'asz Local Lemma}, then the Moser-Tardos algorithm will find a satisfying solution in expected linear time. We re-state their theorem here in order to allow us to define terminology explicitly.

\begin{theorem}
Let $\mathcal{X}$ be a finite set of mutually independent random variables, and $\mathcal{B}$, called bad events, be a finite set of predicates determined by these variables. For $B\in \mathcal{B}$ let $\Gamma(B)$ be a proper subset of $\mathcal{B}$ satisfying that $B$ is independent from the collection of events $\mathcal{B}\setminus (\{B\}\cup\Gamma(B))$. If there exists an assignment  of real numbers $f:\mathcal{B}\to(0,1)$ such that $\forall B\in\mathcal{B}$
$$
 \mbox{\emph{Pr}}(B)\leq f(B)\prod_{A\in\Gamma(B)}(1-f(A)),
$$

\noindent Then there exists an assignment of values to the variables $\mathcal{X}$ such that none of the bad events $\mathcal{B}$ are true. Moreover, Algorithm \ref{alg:Moser-Tardos} resamples an event $B\in\mathcal{B}$ at most an expected $f(B)/(1-f(B))$ times before it finds such an assignment. Thus the expected number of resampling steps is at most $\sum\limits_{B\in\mathcal{B}} f(B)/(1-f(B))$.
\end{theorem}
\noindent A proof of this theorem can be found in \cite{MoserTardos}.\\

\begin{theorem}
\label{candklemma}
\noindent If values of $c$ and $k$ are selected such that 
$$\frac{1}{c}\leq \Big(\frac{1}{k^2+2k-1}\Big)^{\frac{1}{k-1}}\Big(1-\frac{1}{k^2+2k-1}\Big)^{\frac{2k-2+k^2}{k-1}}$$
holds, then the existence of a stable Candy-Crush grid is guaranteed through the Lov\'asz Local Lemma.
\end{theorem}

\begin{proof}
We define $\mathcal X$ as the set $\mathcal X=\{\mathcal X_1,\mathcal X_2,\cdots,\mathcal X_T\}$, where $\mathcal X_i$ is the colour of the $i$th vertex. We define each \textit{bad event} $B_i\in\mathcal{B}$ as the event that the $i$th hyperedge is monochromatic. Using these definitions, we can place bounds on how many events in $\mathcal{B}$ depend on each other. Two events $B_i$ and $B_j$ are dependent if the $i$th and $j$th hyperedge intersect. For each hyperedge $e\in E$ in the Candy-Crush grid, there are at most $2k-2$ hyperedges that intersect $e$ and are parallel to it, where ``parallel" means that both hyperedges are either both horizontal or both vertical. Furthermore, for each hyperedge $e\in E$, there are exactly $k$ vertices in $e$ and there are at most $k$ hyperedges that are perpendicular to $e$ and intersect it. Thus there is a maximum of $2k-2+k^2$ possible hyperedges that can intersect $e$.\\
\newline\noindent
Now that we have a tight upper bound for the number of events in $\mathcal{B}$ that intersect with each other, we can apply the Lov\'asz Local Lemma to prove the existence of $k$-stable $c$-coloured Candy-Crush grids. The Lov\'asz Local Lemma states that a stable Candy-Crush grid exists when, for $B\in\mathcal{B}$,

$$
\text{Pr}(B)\leq f(B)\prod_{A\in\Gamma(B)}(1-f(A)).
$$ 

\noindent For simplicity, we assume $f(B)$ is the same for all $B\in \mathcal{B}$, let us denote this by $f(B)=x$, where $0\leq x\leq1$. Also we assume that $X_i\in \mathcal{X}$ is uniformly distributed in $\{1,\cdots,c\}$. Therefore, we have the inequalities

\begin{equation} \label{LLLeq}
\text{Pr}(B)=1/c^{k-1}\leq f(B)\prod_{A\in\Gamma(B)}(1-f(A))\leq x(1-x)^{2(k-1)+k^2}.
\end{equation}

\noindent We can find an upper bound the to number of colours $c$ that guarantees the existence of a stable Candy-Crush grid by finding the maximum of $x(1-x)^{2(k-1)+k^2}$ for $x\in[0,1]$. Taking the derivative with respect to $x$, we have
$$
\frac{\partial}{\partial x} x(1-x)^{2(k-1)+k^2}=(1-x)^{2k+k^2-3}(1-x(2k-1+k^2)).
$$

\noindent This partial derivative is equal to zero when $x=1$ or $(1-x(2k-1+k^2))=0$. If $x=1$, we have that $1/c^{k-1}\leq 0$, which would imply that no stable Candy-Crush grids exist, thus $x=1$ corresponds to a local minimum. Since $x(1-x)^{2(k-1)+k^2}=0$ when $x=1$ and $x=0$, and since $\frac{\partial}{\partial x} x(1-x)^{2(k-1)+k^2}>0$ for $x=0$, the function $x(1-x)^{2(k-1)+k^2}$ must have a local maximum in the interval $[0,1]$. The only possible value for this maximum is therefore $x=1/(k^2+2k-1)$. Substituting this into (\ref{LLLeq}) we have
\begin{equation}\nonumber
\begin{split}
\frac{1}{c^{k-1}} &\leq \frac{1}{k^2+2k-1}\Big(1-\frac{1}{k^2+2k-1}\Big)^{2k-2+k^2},\\
\end{split}
\end{equation}

\noindent or equivalently,

\begin{equation}
\label{eqnconditions}
\frac{1}{c}\leq \Big(\frac{1}{k^2+2k-1}\Big)^{\frac{1}{k-1}}\Big(1-\frac{1}{k^2+2k-1}\Big)^{\frac{2k-2+k^2}{k-1}}.\tag{\ref{candklemma}}
\end{equation}

\end{proof}

\noindent Note that, due to the results in \cite{MoserTardos}, since a stable Candy Crush grid is guaranteed to exist due to the Lov\'asz Local Lemma the Moser-Tardos algorithm is guaranteed to find a stable grid in expected linear time.\\
\newline\noindent
It follows from Lemma \ref{candklemma} that in the case where $k=3$,  (\ref{eqnconditions}) holds when $c> 6$. Candy-Crush is typically played with $k=3$ and $c\leq6$, so unfortunately our theoretical results do not apply to the game itself. A plot of values of $c$ which satisfy (\ref{eqnconditions}) plotted against $k$ is shown in the white region in Figure \ref{fig:inequality3}. The red region in Figure \ref{fig:inequality3} represents the values of $c$ and $k$ such that we cannot prove that Algorithm \ref{alg:Moser-Tardos} will find a weak $c$-colouring of $H$ in randomised polynomial time. However, we implemented Algorithm \ref{alg:MSAforCC} using Matlab and found that the algorithm appears to run in randomised polynomial time for all values of $c$ when $k\geq3$. This agrees with the observations in Catarata \textit{et al.} \cite{Catarata} in which he found that the Moser-Tardos algorithm can find solutions to a large number of constraint satisfaction problems even though the Lov\'asz Local Lemma could not guarantee the existence of a solution, \textit{i.e.}, $|W'|<|W|$.\\

\noindent Note that there are other methods of defining variables and bad events such that Moser-Tardos will still find a stable Candy-Crush grid. These other formulations of the Candy-Crush problem are non-trivially different, each implying different things about the algorithm's efficiency and output distribution.\\

\begin{figure}
\centering
\begin{tikzpicture}
  \node (img1)  {\includegraphics[trim={2cm 8cm 0 7cm},clip,width=0.6\textwidth]{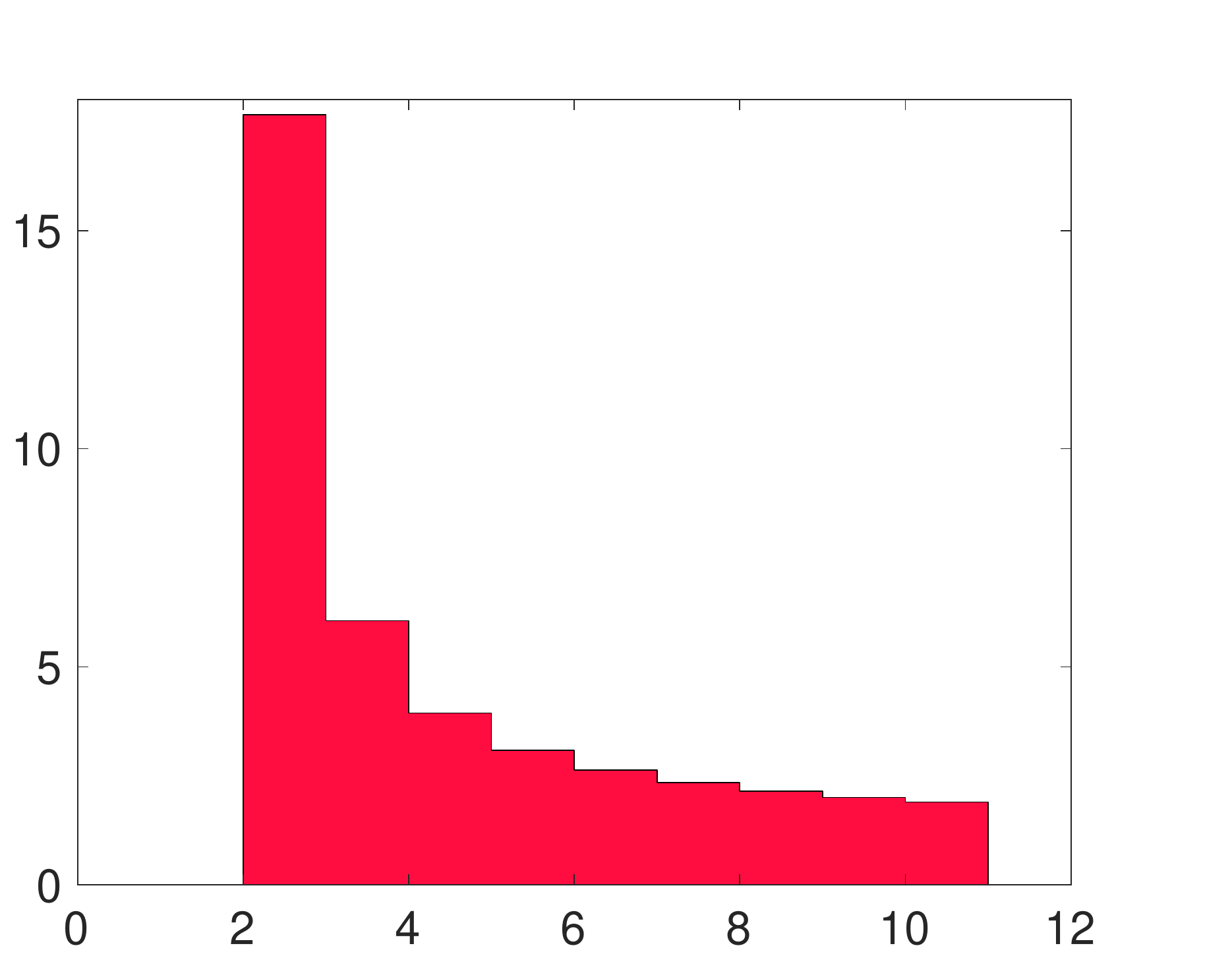}};
  \node[below=of img1, node distance=0cm, yshift=1cm,font=\color{black}] {size of each hyperedge $k$};
  \node[left=of img1, node distance=0cm, rotate=90, anchor=center,yshift=-0.7cm,font=\color{black}] {number of colours $c$};
\end{tikzpicture}

\caption{A plot of the values of values of $c$ and $k$ such that Algorithm \ref{alg:Moser-Tardos} is not guaranteed to find a stable Candy-Crush grid in randomised polynomial time (shown in red). All other values of $c$ and $k$ guarantee that the Moser-Tardos algorithm will find a weak proper $c$-colouring of a $k$-uniform Candy Crush grid in expected polynomial time.}
\label{fig:inequality3}
\end{figure}

\label{complexity}

\subsection{Uniformity of the Moser-Tardos algorithm applied to Candy-Crush}
When the Moser-Tardos algorithm is finished, there is no guarantee that the weak proper $c$-colouring of a Candy Crush grid will be sampled from the uniform distribution. This is a problem as the Multilevel Splitting Algorithm \ref{alg:MSAforCC} will only provide unbiased estimates if the samples from $Y_t$ are drawn uniformly at random. Guo \textit{et al.} \cite{UniformMT} proved that when a constraint satisfaction problem is \textit{extremal} then the output of the Moser-Tardos algorithm is uniformly distributed. A constraint satisfaction problem is said to be extremal if and only if every pair of bad events $B_i,B_j\in \mathcal{B}$ are either mutually exclusive or independent.\\
\newline\noindent
The problem of finding a weak proper $c$-colouring of a hypergraph $H=(V,E)$ is extremal only if all hyperedges in $E$ are disjoint. Since the hyperedges in a nontrivial Candy Crush grid are not disjoint, the result of Guo \textit{et al.} \cite{UniformMT} does not apply here. In this section, we prove that the output of the Moser-Tardos algorithm is uniformly distributed when applied to finding weak proper $c$-colourings of Candy Crush grids.\\

\label{uniformity proof}
\begin{theorem}
Given a uniformly distributed $c$-coloured Candy-Crush grid as an input, the output of the Moser-Tardos algorithm is uniformly distributed across stable Candy crush grids.
\end{theorem}

\begin{proof}
\label{UniformProof}

The critical idea behind this proof is that we resample a larger set of Candies rather than just the candies in a monochromatic hyperedge. We show that this choice of sampling means that the result of Guo \textit{et al.} \cite{UniformMT} will hold, and that the Moser-Tardos algorithm will give uniformly distributed outputs. We then show that resampling the vertices outside of the hyperedge does not affect the distribution of the colours of vertices outside of the hyperedge. This means that if we did not resample any vertices outside of the hyperedge, then the distribution of the Moser-Tardos's output would not change.\\
\newline\noindent 
Here, the input grid is a subset of an $m\times n$ grid, where $m$ and $n$ are natural numbers. We can show that the output of the Moser-Tardos algorithm is uniformly distributed by imposing a lexicographic ordering on the $k$-uniform hyperedges in the grid. In the typical framework of the Moser-Tardos algorithm, the variable $X_i$, for $i \in \{1,2,\cdots ,|V|\}$, is the colour of the $i$th vertex and the bad event $B_j$, for $j \in\{1,2,\cdots,|E|\}$, is the event that the $j$th hyperedge is monochromatic. The typical Moser-Tardos algorithm works by randomly selecting a monochromatic hyperedge and, uniformly and independently, re-sampling all the colours of the vertices in this hyperedge.\\
\newline\noindent
Let us instead consider a different variation of the same problem. In this version, we still define the variable $X_i$ to be the colour of the $i$th vertex. However, we define each bad event $A_j$ to be the event corresponding to the hyperedge indexed by $j$ in the lexicographic ordering is the lexicographically first monochromatic hyperedge. By ``lexicographically first", we mean the monochromatic hyperedge whose lexicographically lowest vertex is lexicographically lower than the lowest vertices in any other monochromatic 3-mer. For brevity we will omit the word ''lexicographically". In the event that two monochromatic hyperedges have the same vertex as their lowest, \textit{i.e.}, both a vertical and horizontal hyperedge, the horizontal hyperedge is considered to be the lowest. This means that each bad event $A_j$ is now a function of the random variables $X_i$ that are included the monochromatic hyperedge, as well as the colour of every other vertex that is lexicographically lower than this hyperedge. \\
\newline\noindent
This choice of indexing the hyperedges implies three things. First, the degree of the dependency graph is now unbounded and the Lov\'asz Local Lemma can no longer be made to hold. This unfortunately means that, in this setup, we would no longer be able to claim that the algorithm runs in expected linear time, but we are not suggesting this as an actual algorithm. Second, the Moser-Tardos algorithm will still output a stable Candy-Crush grid. This is because if there are no lexicographically first monochromatic hyperedges, there are no monochromatic hyperedges. Third, the events $A_j$ are now extremal, because there cannot be two monochromatic hyperedges that are both lexicographically first. This implies that when the Moser-Tardos algorithm is applied to this problem with an input generated uniformly at random, the output will be distributed uniformly at random from the set of satisfying solutions.\\
\newline\noindent
A step of the Moser-Tardos algorithm works by going through every hyperedge in lexicographic order until a monochromatic hyperedge is found. Once this monochromatic hyperedge is found, the colours of the vertices of the monochromatic hyperedge and every lower-ordered vertex are resampled uniformly independently. After each resampling, one of two things can happen:
\begin{enumerate}
\item There are no monochromatic hyperedges in the resampled vertices, and the first monochromatic hyperedge is higher than the one found in the previous step. 
\item There is at least one monochromatic hyperedge that appears after the resampling. This means that the first monochromatic hyperedge is now lower than the one found in the previous step.
\end{enumerate}

\noindent Thus, during the run of the algorithm the position of the first monochromatic hyperedge is changing randomly, either moving up or down. Let us denote the position of the first monochromatic hyperedge before each resampling by $y$. By the position of the hyperedge, we mean the location of its lowest vertex. The algorithm finishes when this position cannot go any higher, \textit{i.e.}, the first monochromatic 3-mer has ``left the screen" or when $y$ is larger than $nm$, and we have a $k$-stable $c$-colouring. Let us denote the set of all the vertices lower than $y$ by $C_y$. Since all the vertices in $C_y$ are uniformly sampled from the set of stable configurations, if we do not resample any of the vertices in $C_y$ we will not change the distribution of the output of the algorithm.\\
\newline\noindent
So we can instead consider a new version of the Moser-Tardos algorithm in which we do not resample the vertices in $C_y$. This means we only resample the vertices that comprise the monochromatic hyperedge, and still achieve a uniform sampling of stable Candy-Crush grids at the end of the algorithm's run. However, this algorithm is simply the original Moser-Tardos algorithm in which we resample only the vertices in monochromatic hyperedges. Therefore, for the original version of the Moser-Tardos algorithm, the output of the algorithm is uniformly distributed from the set of stable Candy-Crush grids.
\end{proof}

\section{Application to general hypergraphs}
Here we, explore the use of Algorithm \ref{alg:MSAforCC} in approximating the weak chromatic polynomials of general hypergraphs.\\
\newline\noindent
We showed that Algorithm \ref{alg:MSAforCC} is an FPRAS for approximating the weak chromatic polynomial of Candy Crush grids when (\ref{eqnconditions}) is satisfied. There are many features of the Candy Crush that are superfluous in proving that Algorithm \ref{alg:MSAforCC} is an FPRAS. For instance, in Section \ref{uniformity proof} the structure of the hypergraph is almost never mentioned. The main requirement for the proof to work is that one can establish a lexicographic order on the hyperedges and vertices of the hypergraph. In Section \ref{Bounding Levels}, we showed that the conditional probabilities $c_t$ are bounded by using the fact that the degree of each vertex is linear in~$k$.\\
\newline\noindent
It is reasonable to hypothesise that Algorithm \ref{alg:MSAforCC} works as an FPRAS for other hypergraphs. These hypergraphs would need to exhibit the same features as the Candy Crush grids, such as low-degree vertices, hyperedges that are bounded in size, and the potential for weak proper $c$-colourings to exist due to the Lov\'asz Local Lemma for sufficiently large~$c$.\\
\newline\noindent
The highly structured nature of Candy Crush grids makes them an ideal structure to which we can apply the Lov\'asz Local Lemma. For general hypergraphs, proving that the Moser-Tardos algorithm runs in expected linear time may be quite difficult, especially for more complex hypergraphs or families of hypergraphs.\\
\newline\noindent
We believe that Algorithm \ref{alg:MSAforCC} would have applications in estimating the weak chromatic polynomials of a wide variety of hypergraphs beyond Candy Crush grids. Determining the types of graphs for which Algorithm \ref{alg:MSAforCC} is efficient remains an open problem.

\section{Conclusion}
In conclusion, we have constructed an FPRAS to estimate the state space of the game Candy Crush. Algorithm \ref{alg:MSAforCC} is one such FPRAS which works when Inequality (\ref{eqnconditions}) holds. When (\ref{eqnconditions}) does not hold, the FPRAS may run either in expected polynomial time or expected exponential time. Determining what conditions are necessary and sufficient for the algorithm to run in expected polynomial time is still an open problem. Empirical results suggest that $|W'|<|W|$. Thus, in the future, we plan to seek tighter bounds.

\subsection*{Acknowledgements}

The second author would like to also acknowledge the support of her ARC DP180103106.


\bibliographystyle{unsrtnat}
\bibliography{references}

\end{document}